\documentclass{amsart}[11pt]
\usepackage{amssymb,mathrsfs,amscd,tikz,graphicx}
\newcommand{\abs}[1]{\lvert #1 \rvert}
\usetikzlibrary{matrix,arrows}
\DeclareMathSymbol{\twoheadrightarrow} {\mathrel}{AMSa}{"10}

\def\Q{{\mathbf Q}}
        \def\PP{{\mathcal P}}

\def\Z{{\mathbf Z}}
\def\C{{\mathbf C}}
\def\R{{\mathbf R}}

\def\H{{\mathrm H}}
\def\ST{{\mathbf S}}

\def\Sn{{\mathbf S}_n}
\def\An{{\mathbf A}_n}

\def\Gal{\mathrm{Gal}}

\def\Lie{\mathrm{Lie}}

\def\End{\mathrm{End}}

\def\Aut{\mathrm{Aut}}
\def\Hom{\mathrm{Hom}}

\def\GL{\mathrm{GL}}

                \def\sL{\mathfrak{sl}}

        \def\K_a{\bar{K}}

\def\E{\mathrm{E}}
\def\dim{\mathrm{dim}}

\def\c{{\mathfrak c}}

                          \def\k{{\mathfrak k}}

                           \def\Hdg{\mathrm{Hdg}}

                      \def\hdg{\mathrm{hdg}}
                            \def\Tr{\mathrm{Tr}}
                            \def\U{\mathrm{U}}
                            \def\f{{\mathfrak f}}
\newtheorem{thm}{Theorem}
\newtheorem{lem}[thm]{Lemma}
\newtheorem{cor}[thm]{Corollary}

\theoremstyle{definition}

\newtheorem{rem}[thm]{Remark}

\newtheorem{step}{Step}

\title[Hodge groups]{Hodge Groups of certain superelliptic jacobians II}
\author{Jiangwei Xue}

\address{National Center for Theoretical Sciences, Math. Division,
  Third General Building, National Tsing Hua University, No.101, Sec
  2, Kuang Fu Road, Hsinchu, Taiwan 30043, Taiwan R.O.C.}

\email{xue\_j\char`\@math.cts.nthu.edu.tw}

\begin{document}
\maketitle
\section{introduction}
Throughout this paper $\C$ is the field of complex numbers,
$K\subseteq \C$ is a subfield of $\C$, $f(x)\in K[x]$ a polynomial
without multiple roots and of degree $n\geq 4$. Let $p\in \mathbf{N}$
be a prime that does not divide $n$ and $q=p^r\in \mathbf{N}$ an
integral power of $p$. We write $C_{f,q}$ for the superelliptic
$K$-curve $y^q=f(x)$, and $J(C_{f,q})$ for the Jacobian of
$C_{f,q}$.  By definition, $C_{f,q}$ is the smooth projective model of
the affine curve $y^q=f(x)$. The Jacobian $J(C_{f,q})$ is an abelian
variety over $K$ of dimension
\[ \dim J(C_{f,q})= g(C_{f,q})=\frac{(n-1)(q-1)}{2}. \]
If $q>p$, the map 
\[ C_{f,q}\to C_{f,q/p}, \quad (x,y)\mapsto (x,y^p)\] induces by
Albanese fuctoriality a surjective $K$-map between the Jacobians
$J(C_{f,q})\to J(C_{f,q/p})$.  We write $J^{(f,q)}$ for the identity
component of the kernel. If $q=p$, we set $J^{(f,p)}=J(C_{f,p})$.  It
is follows easily that $J^{(f,q)}$ is an abelian variety over $K$ of
dimension $(n-1)\varphi(q)/2$, where $\varphi$ denotes the Euler
$\varphi$-function. Moreover, $J(C_{f,q})$ is $K$-isogenous to the
product $\prod_{i=1}^r J^{(f,p^i)}$(See \cite{ZarhinM}).

Since $K\subseteq \C$, we may view $J^{(f,q)}$ as a complex abelian
variety. We refer to \cite{Ribet3}, \cite[Sect. 6.6.1 and
6.6.2]{ZarhinIzv} for the definition and basic properties of the Hodge
group (aka special Mumford--Tate group). In \cite{xuezarhin2},
assuming that $n>q$ and some other conditions on $n, q$ and $f(x)$,
the authors showed that the (reductive $\Q$-algebraic connected) Hodge
group of $J^{(f,q)}$ coincides with the largest $\Q$-algebraic
subgroup of $\GL(\H^1(J^{(f,q)},\Q))$ that's ``cut out'' by the
induced polarization from the canonical principal polarization of
$J(C_{f,q})$ and the endomorphism ring of $J^{(f,q)}$. Notice that
when $q=2$ (i.e., in the hyperelliptic case) this group was completely
determined in \cite{ZarhinMMJ} (when $f(x)$ has ``large" Galois
group). In this paper, we study some additional properties of
$J^{(f,q)}$ which will allow us to extend the result to the case $n<q$
as well. This case is necessary in order to treat the infinite towers
of superelliptic jacobians, which, in turn, are useful for the study
of the ranks of Mordell-Weil groups in infinite towers of function
fields (See \cite{UlmZarhin}). 

To state our main result, we make explicit the endomorphism ring and
the polarization mentioned above. Let $X$ be an abelian variety over
$\bar{K}$. We write $\End(X)$ for the ring of all its
$\bar{K}$-endomorphisms and $\End^0(X)$ for the endomorphism algebra
$\End(X)\otimes_\Z \Q$. In a series of papers
\cite{ZarhinMRL,ZarhinCrelle,ZarhinCamb,ZarhinM}, Yuri Zarhin
discussed the structure of $\End^0(J(C_{f,q}))$, assuming that $n\ge
5$ and the Galois group $\Gal(f)$ of $f(x)$ over $K$ is, at least,
doubly transitive. Here $\Gal(f)\subseteq \ST_n$ is viewed as a
permutation group on the roots of $f(x)$. It is well known that $f(x)$
is irreducible over $K$ if and only if $\Gal(f)$ acts transitively on
the roots. For the sake of simplicity let's assume that $K$ contains a
primitive $q$-th root of unity $\zeta_q$. The curve $C_{f,q}:
y^q=f(x)$ admits the obvious periodic automorphism
\[ \delta_q: C_{f,q}\to C_{f,q}, \quad (x,y)\mapsto (x,\zeta_q y).\]
By an abuse of notation, we also write $\delta_q$ for the induced
automorphism of $J(C_{f,q})$. The subvariety $J^{(f,q)}$ is
$\delta_q$-invariant and we have an embedding
\[\Z[\zeta_q]\hookrightarrow \End(J^{(f,q)}), \quad \zeta_q\mapsto
\delta_q.\] In particular, the $q$-th cyclotomic filed
$E:=\Q(\zeta_q)$ is contained in $\End^0(J^{(f,q)})$. Zarhin showed
(\cite{ZarhinMRL,ZarhinM,ZarhinMZ2}) that $\End(J^{(f,q)})$ is isomorphic to
$\Z[\zeta_q]$ if either $\Gal(f)$ coincides with the full symmetric
group $\ST_n$, $n\geq 4$ and $p\geq 3$, or $\Gal(f)$ coincides with
the alternating group $\An$ (or $\ST_n$), and $n\geq 5$. This result has also been
extended to the case $\Gal(f)=\ST_n$ or $\An$, $n\geq 5$ and $p\mid n$
in \cite{xuejnt}.

The first rational homology group $\H_1(J^{(f,q)}, \Q)$ carries a
natural structure of $E$-vector space of dimension
\[\dim_E\H_1(J^{(f,q)},\Q)=\frac{\dim_\Q
  \H_1(J^{(f,q)},\Q)}{[E:\Q]}=\frac{2\dim
  J^{(f,q)}}{[E:\Q]}=\frac{(n-1)\varphi(q)}{\varphi(q)}=n-1.\] Notice
that if $q>2$, then $E$ is a CM field with complex conjugation $e\mapsto
\bar{e}$. Let
\[E^{+}=\{e\in \Q(\zeta_q) \mid \bar{e}=e\}\] be the maximal totally
real subfield of $E$ and let
\[E_{-}=\{e\in \Q(\zeta_q) \mid \bar{e}=-e\}.\] % The induced
% polarization on $J^{(f,q)}$ from the canonical principal polarization
% on $J(C_{f,q})$ gives rise to 
The canonical principal polarization on $J(C_{f,q})$ induces a
polarization on $J^{(f,q)}$, which gives rise to a nondegenerate
$E$-sesquilinear Hermitian form (\cite{xuezarhin2})
\[\phi_q:\H_1(J^{(f,q)},\Q) \times \H_1(J^{(f,q)},\Q) \to
E. \] We write $\U(\H_1(J^{(f,q)},\Q),\phi_q)$ for the
unitary group of $\phi_q$ of the $\Q(\zeta_q)$-vector space
$\H_1(J^{(f,q)},\Q)$, viewed as an $\Q$-algebraic subgroup of
$\GL(\H_1(J^{(f,q)},\Q))$ (via Weil's restriction of scalars from
$E^{+}$ to $\Q$ (\cite{Ribet3})). Since the Hodge group respects
the polarization and commutes with endomorphisms of $J^{(f,q)}$,
\[\Hdg(J^{(f,q)})\subset \U(\H_1(J^{(f,q)},\Q),\phi_q).\]
If $\End^0(J^{(f,q)})=E$, then $\U(\H_1(J^{(f,q)},\Q),\phi_q)$ is the
largest connected reductive $\Q$-algebraic subgroup of
$\GL(\H_1(J^{(f,q)},\Q))$ that both respects the polarization and
commutes with endomorphisms of $J^{(f,q)}$.

The following theorem is a natural extension of \cite[Theorem
0.1]{xuezarhin2}. 
\begin{thm}
\label{main} Suppose that  $n\ge 4$ and   $p$ is a prime that does not divide
$n$. Let $f(x) \in \C[x]$ be a degree $n$ polynomial without multiple roots.
Let $r$ be a positive integer and $q=p^r$. Suppose that there exists a subfield
$K$ of $\C$ that contains all the coefficients of $f(x)$. Let us assume that
$f(x)$ is irreducible over $K$ and the Galois group $\Gal(f)$ of $f(x)$ over
$K$ is either $\Sn$ or  $\An$. Assume additionally that either $n \ge 5$ or
$n=4$ and $\Gal(f)=\ST_4$.

Suppose that one of the following three conditions holds:
\begin{itemize}
\item[(A)] $n=q+1$;

\item[(B)] $p$ is odd and $n\not\equiv 1 \mod q$;

\item [(C)] $p=2$, $n\not\equiv 1 \mod q$ and $n\not \equiv q-1 \mod 2q$.
\end{itemize}

Then $\Hdg(J^{(f,q)})= \U(\H_1(J^{(f,q)},\Q),\phi_q)$.
\end{thm}

\begin{cor}
  Corollary~0.3, Theorem 4.2 and Theorem 4.3 of \cite{xuezarhin2} all
  hold without the assumption that $n>q$.
\end{cor}

\begin{rem}\label{sec:n_less_than_q}
  We assume that $n<q$ throughout the rest of the paper since the case
  $n>q$ has already been treated in \cite{xuezarhin2}.
 \end{rem}

\begin{rem}
  Since both $\Hdg(J^{(f,q)})$ and $\U(\H_1(J^{(f,q)},\Q),\phi_q)$ are
  connected $\Q$-algebraic groups, to prove Theorem~\ref{main}, it
  suffices to show that
\[ \dim \Hdg(J^{(f,q)})\geq \dim
\U(\H_1(J^{(f,q)},\Q),\phi_q).\]
It is known that 
\[ \dim\U(\H_1(J^{(f,q)},\Q),\phi_q) =\dim_\Q E^{+}\cdot
\big(\dim_E\H_1(J^{(f,q)},\Q)\big)^2.\]
%=\frac{1}{2}[E:\Q]\big(\dim_E\H_1(J^{(f,q)},\Q)\big)^2.\]
Let $\hdg$ be the $\Q$-Lie algebra
of $\Hdg(J^{(f,q)})$. It is a reductive $\Q$-Lie subalgebra of
$\End_\Q\big(\H_1(J^{(f,q)},\Q)\big)$, and thus splits into a direct
sum 
\[ \hdg= \c\oplus \hdg^{ss}, \] of its center $\c$ and the semisimple
part $\hdg^{ss}=[\hdg, \hdg]$. By \cite[Theorem 1.3]{xuezarhin1}, if
$\Gal(f)=\ST_n$ and $n\geq 4$, or $\Gal(f)=\An$ and $n\geq 5$, the
center $\c$ coincides with $E_{-}$. Notice that
\[\dim_\Q E_{-}=\dim_\Q E^{+}=[E:\Q]/2.\]
Theorem~\ref{main} follows if we show that 
\begin{equation}
  \label{eq:1}
\dim_\Q \hdg^{ss}\geq \frac{1}{2}[E:\Q] \big((\dim_E\H_1(J^{(f,q)},\Q))^2-1\big).  
\end{equation}
\end{rem}

The paper is organized as follows.  In section 2 we study the Galois
actions on certain vector spaces.  In section 3 we recall some facts
about the Hodge Lie algebra $\hdg$. The proof of Theorem~\ref{main} is
given at the end of section 3 except a key arithmetic lemma, which is
proven in Section 4.

\section{Galois Actions}

Throughout this section, let $E$ be a field that is a finite Galois
extension of $\Q$ with Galois group $G$. Let $V$ be a $E$-vector space
of finite dimension. We write $V_\Q$ for the underlying $\Q$-vector
space of $V$, and $V_{\C}$ for the $\C$-vector space $V\otimes_\Q
\C=V_\Q\otimes_\Q \C$. Let $\Aut(\C)$ be the group of all
automorphisms of $\C$.  It act semilinearly on $V_{\C}=V\otimes_\Q \C$
through the second factor. More explicitly, $\forall \kappa \in
\Aut(\C), v\otimes z\in V\otimes_\Q \C$, we define $\kappa (v\otimes
z):=v\otimes \kappa(z)$. It follows that $\forall x\in V\otimes_\Q \C$
and $c\in \C$, $\kappa(cx)=\kappa(c)x$.  On the other hand, $E$ acts
on $V_\C=V\otimes_\Q \C$ through its first factor. It follows that
$V_\C$ is a free $E\otimes_\Q\C$ module of rank $\dim_E V$, and the
action of $E=E\otimes 1\subseteq E\otimes_\Q\C$ commutes with that of
$\Aut(\C)$. In other words,
\[\kappa((e\otimes 1)x)=(e\otimes 1)\kappa(x), \quad \forall \kappa
\in \Aut(\C),  e\in E, \text{ and } x\in V_\C.\]

Let's fix an embedding $E\hookrightarrow \C$. This allows us to
identify each Galois automorphism $\sigma: E\to E$ with the embedding
$\sigma: E\to E \subset \C$ of $E$ into $\C$.
It is well known that
\[ E_\C:=\E\otimes_\Q \C =\bigoplus_{\sigma\in G}
E\otimes_{E,\sigma}\C=\bigoplus_{\sigma\in G} \C_{\sigma}, \text{
  where } \C_\sigma:=E\otimes_{E,\sigma}\C. \] So every $E_\C$ module
$W$ splits as a direct sum $ W=\oplus_{\sigma\in G} W_\sigma$, where
\[ W_\sigma:=\C_\sigma W=\{w\in W\mid (e\otimes 1)w= \sigma(e)w,
\forall e\in E\}.\] In particular, $V_\C=\oplus_{\sigma\in G}
V_\sigma$, and each $V_\sigma$ is a $\C$-vector space of dimension
$\dim_E V$. For each $\sigma \in G$, let $P_\sigma: V_\C\to V_\sigma$
be the $\C$-linear projection map from $V_\C$ to the summand
$V_\sigma$. Similarly, for each pair $\sigma\neq \tau$, we write
$P_{\sigma, \tau}=P_\sigma\oplus P_\tau: V_\C\to V_\sigma\oplus
V_\tau$ for the projection map onto this pair of summands.

We claim that $\Aut(\C)$ permutes the set $\{V_\sigma \mid \sigma\in
G\}$, and the action factors through the canonical
restriction 
\[\Aut(\C)\twoheadrightarrow G,\quad \kappa \mapsto \kappa\mid_E.\]  Indeed, for all $\kappa
\in \Aut(\C), e\in E$ and $x_\sigma\in V_\sigma$, 
\[ (e\otimes 1)\kappa (x_\sigma)= \kappa ((e\otimes
1)x_\sigma)=\kappa( \sigma(e)x_\sigma)=
\kappa(\sigma(e))\kappa(x_\sigma)=\kappa\sigma(e)\kappa(x_\sigma).\]
Clearly $\kappa\sigma(e)=((\kappa\mid_E) \sigma)(e)$. By an abuse of
notation, we write $\kappa$ for the restriction $\kappa\mid_E$.
So it follows that $\kappa(x_\sigma)\in V_{\kappa\sigma}$, and thus
$\kappa(V_\sigma)=V_{\kappa\sigma}$ for all $\kappa \in \Aut(\C)$ and
$\sigma\in G$. 

% Clearly the action factors through the canonical
% restriction $\Aut(\C)\to G, \kappa \mapsto \kappa\mid_E$.

Let us define an action of $\Aut(\C)$ on the set of projection
$\PP=\{ P_\sigma \mid \sigma\in G\}$ by
\[ \kappa_* P_\sigma:= \kappa \circ P_\sigma \circ \kappa^{-1}.\] Then
for any element $\sum  x_\sigma \in \oplus_{\sigma\in G}
V_\sigma=V_\C$ and $P_\tau\in \PP$,
\[ (\kappa_* P_\tau)(\sum x_\sigma)= \kappa\circ P_\tau\left(\sum
  \kappa^{-1}(x_\sigma)\right)=\kappa (\kappa^{-1}(x_{\kappa
  \tau}))=x_{\kappa \tau},\] where all summations runs through
$\sigma\in G$, and we used the fact that $\kappa^{-1}(x_\sigma)$
belongs to $V_{\tau}$ if and only if $\sigma= \kappa\tau$. Therefore,
\[\kappa_*P_\sigma=P_{\kappa\sigma}.\]
Clearly $\Aut(\C)$ acts transitively on $\PP$.  Since
$P_{\sigma,\tau}=P_\sigma\oplus P_\tau$, we have similarly an action
of $\Aut(\C)$ on the set $\PP\PP:=\{P_{\sigma,\tau}\mid (\sigma,
\tau)\in G^2, \sigma\neq \tau\}$ by 
\[ \kappa_* P_{\sigma, \tau}=\kappa\circ P_{\sigma, \tau}\circ
\kappa^{-1}=P_{\kappa\sigma, \kappa \tau}.\] The $\Aut(\C)$-orbit
$O_{\sigma,\tau}$ of each $P_{\sigma, \tau} \in \PP\PP$ consists of
all elements of the form $P_{\kappa\sigma, \kappa\tau}$ with
$\kappa\in G$.

% Let us identify the set $\PP$ with $G$, and similarly $\PP\PP$ with
% $G^2-\Delta$. 

\begin{lem}\label{lem:galois-actions}
  Let $W_\Q\subseteq V_\Q$ be any $\Q$-subspace of $V_\Q$, and
  $W_\C:=W_\Q\otimes_\Q\C\subseteq V_\C$ be its complexification. 
  \begin{itemize}
  \item[(i)] If there exists $\sigma_0\in G$ such that
    $P_{\sigma_0}(W_\C)=V_{\sigma_0}$, then
    $P_{\sigma}(W_\C)=V_{\sigma}$ for all $\sigma\in G$.
  \item[(ii)] If there exists a pair $(\sigma_0,\tau_0)\in G^2$ with
    $\sigma_0\neq \tau_0$
    such that $P_{\sigma_0,\tau_0}(W_\C)=V_{\sigma_0}\oplus
    V_{\tau_0}$, then $P_{\sigma,\tau}(W_\C)=V_{\sigma}\oplus V_\tau$
    for all $P_{\sigma,\tau}\in O_{\sigma_0,\tau_0}$.
  \end{itemize}
\end{lem}
\begin{proof}
  Clearly, $W_\C$ is $\Aut(\C)$-invariant. For each $\sigma\in G$,
  let us choose $\kappa \in \Aut(\C)$ such that $\sigma=\kappa
  \sigma_0$. Then 
  \[ P_\sigma(W_\C)= (\kappa_* P_{\sigma_0})(W_\C)= \kappa \circ
  P_{\sigma_0}\circ \kappa^{-1}(W_\C)=\kappa\circ
  P_{\sigma_0}(W_{\C})=\kappa(V_{\sigma_0})=V_\sigma.\] This proves
  part (i). %Part (ii) can be proven similarly.
  Similarly, suppose that
  $P_{\sigma_0,\tau_0}(W_\C)=V_{\sigma_0}\oplus V_{\tau_0}$. For all
  $P_{\sigma,\tau}\in O_{\sigma_0,\tau_0}$, there exists $\kappa \in
  \Aut(\C)$ such that $\sigma=\kappa \sigma_0$ and $\tau=\kappa
  \tau_0$. So we have
\[
\begin{split}
P_{\sigma,\tau}(W_\C)&=(\kappa_*
P_{\sigma_0,\tau_0})(W_\C)=\kappa\circ P_{\sigma_0,\tau_0}\circ
\kappa^{-1}(W_\C)=\kappa \circ P_{\sigma_0,\tau_0}
(W_\C)\\&=\kappa(V_{\sigma_0}\oplus
V_{\tau_0})=\kappa(V_{\sigma_0})\oplus \kappa
(V_{\tau_0})=V_\sigma\oplus V_\tau,
\end{split}
\]
and part (ii) follows. 
\end{proof}

Let $R$ be a commutative ring with unity, and $N$ be a free $R$-module
of finite rank. We write $\Tr_R:\End_R(N)\to R$ for the trace map,
and \[\sL_R(N):=\{ g\in \End_R(N)\mid \Tr_R(g)=0\}\] for the $R$-Lie
algebra of traceless endomorphisms of $N$.
% Let $\Tr_E:\End_E(V)\to E$ be the trace map, and \[\sL_E(V):=\{
% g\in \End_E(V)\mid \Tr_E(g)=0\}\]
%  be the simple $E$-Lie algebra of traceless endomorphisms of $V$. 
It is well-known that 
\[ \sL_E(V)\otimes_\Q\C =
\sL_{E_\C}(V_\C)=\sL_{E_\C}(\oplus_{\sigma\in
  G}V_\sigma)=\bigoplus_{\sigma\in G} \sL_{\C}(V_\sigma).\] We will
denote the projection map $\sL_E(V)\otimes_\Q\C \to \sL_\C(V_\sigma)$
again by $P_\sigma$, and similarly for $P_{\sigma,\tau}$. Clearly,
each $\sL_{\C}(V_\sigma)$ has $\C$-dimension $(\dim_EV)^2-1$.

For the rest of the section, we assume additionally that $E$ is a
CM-field. For any $\sigma\in G$, let $\bar\sigma:E\to E$ be the
complex conjugation of $\sigma$. In other words, $\bar{\sigma}$ is the
composition $E\xrightarrow{\sigma}E\to E$, where the second arrow
stands for the complex conjugation map $e\mapsto \bar{e}$.

\begin{lem}\label{lem:dimension-comparison}
  Let $\k$ be a semisimple $\Q$-Lie subalgebra of $\sL_E(V)$, and
  $\k_\C:=\k\otimes_\Q \C$ be its complexification.  Suppose that the
  following two conditions holds:
\begin{itemize}
\item[(I)] there exists $\sigma_0\in G$ such that
  $P_{\sigma_0}(\k_\C)=\sL_\C(V_{\sigma_0})$;
\item[(II)] For each pair $(\sigma, \tau)\in G^2$ with $\sigma\neq
  \tau$ and $\sigma\neq \bar\tau$, there exists $P_{\sigma_0,
    \tau_0}\in O_{\sigma,\tau}$ such $P_{\sigma_0,
    \tau_0(\k_\C)}=\sL_\C(V_{\sigma_0})\oplus \sL_\C(V_{\tau_0})$.
\end{itemize}
Then \[\dim_\Q \k\geq \frac{1}{2}[E:\Q]\left((\dim_E V)^2-1\right).\]
\end{lem}

% is $\Aut(\C)$-invariant. 

%\[ E\otimes_\Q\C \cong \otimes \C_{\sigma}\]

\begin{proof}
Applying Lemma~\ref{lem:galois-actions} with $\k$ in place of $W$ and
$\sL_E(V)$ in place of $V$, we see that
\begin{gather*}
  P_\sigma(\k_\C)=\sL_\C(V_\sigma), \quad \forall \sigma\in G;\\
  P_{\sigma,\tau}(\k_\C)= \sL_\C(V_\sigma)\oplus\sL_\C(V_\tau), \quad
  \forall (\sigma,\tau)\in G^2 \text{ with }\sigma\neq \tau \text {
    and }\sigma \neq \bar\tau.
\end{gather*}

Let us fix a CM-type $\Phi$ of $E$. By definition, $\Phi$ is a maximal
subset of $G=\Hom(E,\C)$ such that no two elements of $\Phi$ are
complex conjugate to each other. Clearly, $\abs{\Phi}=[E:\Q]/2$, and 
\[ \dim_\C \Big(\bigoplus_{\sigma\in
  \Phi}\sL_\C(V_\sigma)\Big)=\frac{1}{2}[E:\Q](\dim_E(V)^2-1).\] Let
$\k_\C'$ be the projection of $\k_\C$ on $\oplus_{\sigma\in
  \Phi}\sL_\C(V_\sigma)$. It follows that the projection $\k'_\C\to
\sL_\C(V_\sigma)$ is surjective for all $\sigma\in \Phi$, and $\k'_\C$
also projects surjectively onto $\sL_\C(V_\sigma)\oplus
\sL_\C(V_\tau)$ for all distinct pairs $\sigma, \tau\in \Phi$.
Therefore, $\k_\C'=\oplus_{\sigma\in \Phi}\sL_\C(V_\sigma)$ by the
Lemma on pp.790-791 of \cite{Ribet}.  In particular, we get
\[ \dim_\Q \k =\dim_\C \k_\C\geq \dim_\C
\k_\C'=\frac{1}{2}[E:\Q]\left((\dim_E V)^2-1\right).\]
\end{proof}

In the next section, we will show that our semisimple part of Hodge
Lie algebra $\hdg^{ss}=[\hdg, \hdg]$ satisfies (I) and (II) of
Lemma~\ref{lem:dimension-comparison} and thus prove our Main Theorem.
% Recall that the canonical principal polarization on $J(C_{f,q})$ is
% invariant under all automorphisms of the Jacobian that are induced
% from automorphisms of $C_{f,q}$.  In particular, the induced
% polarization on $J^{(f,q)}$ is $\delta_q$-invariant. It follows that
% we have a $\delta_q$-invariant nondegenerate alternating $\Q$-bilinear
% form
% \[ \psi_q: \H_1(J^{(f,q)},\Q)\times \H_1(J^{(f,q)},\Q) \to \Q\]
% on the first rational homology group of $J^{(f,q)}$ associated to this
% polarization. 

\section{the hodge lie algebra }
We keep all notation and assumptions of the previous sections. More
specifically, $\zeta_q$ is a primitive $q$-th root of unity,
$E=\Q(\zeta_q)$ and $G=\Gal(E/\Q)=(\Z/q\Z)^*$, where each $a\in
(\Z/q\Z)^*$ maps $\zeta_q$ to $\zeta_q^a$.  In order to simplify the
notation, we write $X$ for the abelian variety $J^{(f,q)}$, and $V$
for its first rational homology group $\H_1(X, \Q)$.  In addition, we
assume that $\End^0(X)=E$.

Recall that $E_\C=E\otimes_\Q\C$. Let $\Lie(X)$ be the complex tangent space to the origin of $X$. By
functoriality, $E$ acts on $\Lie(X)$ and provides $\Lie(X)$ with a
natural structure of $E_\C$-module. Therefore, $\Lie(X)$
splits into a direct sum \[\Lie(X)=\oplus_{a\in G} \Lie(X)_a.\]
where $\Lie(X)_a:=\{x\in \Lie(X)\mid (\zeta_q\otimes 1) x= \zeta_q^a
x\}$. 
Let us put $n_a=\dim_\C\Lie(X)_a$. It is known that $n_a=[na/q]$
(see \cite{ZarhinM,ZarhinPisa}),
where $[x]$ is the maximal integer that's less or equal to $x$, and we
take the representative $1\leq a \leq q-1$. 

\begin{rem}\label{rem:relative-prime}
  By \cite[Proposition 2,1, 2.2]{xuezarhin2}, the assumptions
  (A)(B)(C) of Theorem~\ref{main} guarantee that there exists an
  integer $a$ such that \[ 1\leq a \leq q-1, \quad \gcd(a, p)=1\] and
  the integers $[na/q]$ and $\dim_EV=n-1$ are relative prime. We note
  that the conditions (A)(B)(C) of Theorem~\ref{main} are equivalent
  to the conditions (A)(B)(C) of \cite[Theorem 0.1]{xuezarhin2}.
\end{rem}

Since $V=\H_1(X,\Q)$ carries a natural structure of $E$-vector space,
the first complex homology group
$V_\C=\H_1(X,\C)=\H_1(X,\Q)\otimes_\Q\C$ carries a structure of
$E_\C$-module, and therefore splits into a direct sum
\[ V_\C= \oplus_{a\in G} V_a. \]
Each $V_a$ is a $\C$-vector space of dimension $\dim_EV= n-1$. 

There is a canonical Hodge decomposition (\cite[chapter 1]{Mumford},
\cite[pp.~52--53]{Deligne})
\[V_\C=\H_1(X,\C)=\H^{-1,0}(X) \oplus \H^{0,-1}(X)\] where
$\H^{-1,0}(X)$ and $\H^{0,-1}(X)$ are mutually ``complex conjugate"
$\dim(X)$-dimensional complex vector spaces. This splitting is
$E$-invariant, and $\H^{-1,0}(X)$ and $\Lie(X)$ are canonically
isomorphic as $E_\C$-modules. In particular, 
\[ \dim_\C \H^{-1,0}(X)_a= \dim_\C \Lie(X)_a= n_a. \]

Let $\f_H^{0}=\f_{H,Z}^{0}:V_\C \to V_\C$ be the $\C$-linear operator
such that 
\[\f_H(x) =-x/2 \quad \forall \ x \in \H^{-1,0}(X); \quad \f_H^{0}(x)=x/2 \quad
\forall \ x \in \H^{0,-1}(X).\] Since the Hodge decomposition is
$E$-invariant, $\f_H^{0}$ commutes with $E$. Therefore, each $V_a$ is
$\f_H^{0}$-invariant. It follows that the linear operator
$\f_H^{0}:V_a\to V_a$ is semisimple and its spectrum lies in the
two-element set $\{-1/2, 1/2\}$. The multiplicity of eigenvalue $-1/2$
is $n_a=\dim_\C \H^{-1,0}(X)_a$, while the multiplicity of eigenvalue
$1/2$ is $\dim_EV-n_a$. Clearly, the complex conjugate of $a\in
\Gal(E/\Q)=(\Z/q\Z)^*$ is $\bar a=q-a$.  It is known (\cite{Deligne},
\cite{MZ}) that
\begin{equation}
  \label{eq:2}
n_a+n_{\bar a}=\dim_E V.  
\end{equation}
This implies that the multiplicity of the eigenvalue $1/2$ is $n_{\bar{a}}$.

The Hodge Lie algebra $\hdg$ of $X$ is a reductive $\Q$-Lie subalgebra
of $\End_\Q(V)$. Its natural representation in $V$ is completely
reducible and its centralizer in $\End_\Q(V)$ coincides with
$\End^0(X)=E$. Moreover, its complexification
\[ \hdg_\C=\hdg\otimes_\Q \C\subset \End_\Q(V)\otimes_\Q\C
=\End_\C(V_\C)\] contains $\f_H^0$
\cite[Sect. 3.4]{xuezarhin1}. Recall that $\hdg=\c\oplus \hdg^{ss}$,
with $\c$ being the center of $\hdg$ and $\hdg^{ss}=[\hdg,\hdg]$ the
semisimple part. Let $\c_\C:=\c\otimes_\Q \C$ be the complexification
of $\c$ and $\hdg^{ss}_\C:=\hdg^{ss}\otimes_\Q\C$ the complexification
of $\hdg^{ss}$. Clearly, $\hdg^{ss}\subset \sL_{E}(V)$, and thus
\[ \hdg^{ss}_\C\subset \sL_{E_\C}(V_\C)=\oplus_{a\in G}
\sL_\C(V_a). \] We write $\hdg^{ss}_a$ for the image of projection
$P_a: \hdg^{ss}_\C\to \sL_\C(V_a)$. Clearly, each $\hdg^{ss}_a$ is a
semisimple complex Lie subalgebra of $\sL_\C(V_a)$. 

\begin{rem}\label{rem:operator-with-two-eigenvalue}
  Let us decompose $f_H^0$ as $f+f'$ with $f'\in \c_\C$ and $f\in
  \hdg^{ss}_\C$. By \cite[Remark 3.2]{xuezarhin2}, the natural
  representation $V_a$ of $\hdg^{ss}_a$ is simple for all $a\in G$. It
  follows from Schur's Lemma that when restricted to each $V_a$, $f'$
  coincides with multiplication by scalar $c_a\in \C$. Therefore,
  $\hdg^{ss}_\C$ contains an operator
  (namely, $f$) whose restriction on each $V_a$ is diagonalizable with
  at most two eigenvalues: $-1/2-c_a$ of multiplicity $n_a$ and
  $1/2-c_a$ of multiplicity $n_{\bar a}=\dim_EV-n_a$. 
\end{rem}

\begin{lem}\label{lem:one-factor}Let the assumptions be the same as in
  Theorem~\ref{main}.
  There exists an $a\in G=(\Z/q\Z)^*$ such that
  $\hdg^{ss}_a=P_a(\hdg^{ss}_\C)$ coincides with $\sL_\C(V_a)$.
\end{lem}
\begin{proof}
  The idea is to combine Remark~\ref{rem:relative-prime},
  ~\ref{rem:operator-with-two-eigenvalue} together with Lemma~3.3 of
  \cite{xuezarhin2}.  This result is already contained in the proof of
  \cite[Theorem 3.4]{xuezarhin2}, where we note that the assumption
  $n>q$ in \cite[Theorem 3.4]{xuezarhin2} is not used for this
  particular step of the proof.
\end{proof}
Notice that this is the place where assumptions
(A)(B)(C) in Theorem~\ref{main} are used, since we need to make sure that there
exists $a\in G$ such that $n_a$ and $\dim_EV$ are relative prime in
order to apply Lemma~3.3 of \cite{xuezarhin2}.

Let $h:(\Z/q\Z)^*\to \R$ be the function such that for all $1\leq a
\leq q-1$ with $\gcd(a, q)=1$, 
\begin{equation}
  \label{eq:3}
h(a)=\left(\frac{\dim_EV}{2}-n_a\right)^2=\left(\frac{n-1}{2}-\left[\frac{na}{q}\right]\right)^2.
\end{equation}
By~(\ref{eq:2}), $n_a+n_{\bar a}=\dim_EV$, so $h(a)=h(\bar a)=h(q-a)$,
which is also easy to check directly from~(\ref{eq:3}).  The function
$h$ is non-increasing on the set of integers \[ [1, q/2]_\Z:=\{a \mid
1\leq a \leq q/2, \gcd(a,p)=1\}.\] By Remark~\ref{sec:n_less_than_q},
we have $4\leq n <q$. In particular, $[n/q]=0$. On the other hand, let
$t$ be the maximal element of $[1,q/2]_\Z$. Then $t\neq 1$ and
$[nt/q]\neq 0$. It follows that $h$ is not a constant function.

\begin{lem}\label{lem:two-factors}
  Let the assumption be the same as Theorem~\ref{main}. Let $(a,b)\in
  G^2$ be a pair such that $h(a)\neq h(b)$. Then
  $P_{a,b}(\hdg^{ss}_\C)=\sL_\C(V_a)\oplus \sL_\C(V_b)$.
\end{lem}
\begin{proof}By (\ref{eq:3}), 
  \[ h(a)-h(b)=(n_a-n_b)(\dim_E V-n_a-n_b). \] So $h(a)\neq h(b)$ if
  and only if $n_a\neq n_b$ and $n_a\neq \dim_EV-n_b$.  Let
  $\k^{ss}=P_{a,b}(\hdg^{ss}_\C)$. By Lemma~\ref{lem:one-factor} and
  part (i) of Lemma~\ref{lem:galois-actions}, both projections
  $\k^{ss}\to \sL_\C(V_a)$ and $\k^{ss}\to\sL_\C(V_b)$ are
  surjective. By Remark~\ref{rem:operator-with-two-eigenvalue},
  $P_{a,b}(f)$ is a semisimple element of
  $\k^{ss}\subseteq \End_\C(V_a)\oplus \End_\C(V_b)$ such that
  $P_{a,b}(f)$ acts on $V_a$ with (at most) 2 eigenvalues of multiplicities
  $n_a$ and $\dim_EV-n_a$ respectively, and similarly for $b$.
  Lemma~\ref{lem:two-factors} follows by setting $d=2$ in \cite[Lemma
  3.6]{xuezarhin2}.  Last, we point out that the assumption that the
  multiplicities $a_i$ are positive in \cite[Lemma 3.6]{xuezarhin2} is
  not used in its proof, so the lemma applies to the case that $n_a$
  or $n_b$ is zero, which may happen if $n<q$.
\end{proof}

\begin{proof}[Proof of Theorem~\ref{main}]
  % If $q=3$ or $4$, then $[E:\Q]=2$ and Theorem~\ref{main} follows
  % directly from Lemma~\ref{lem:one-factor}. So we further assume that
  % $q>4$.
  As remarked at the end of Section 2, Theorem~\ref{main} follows if
  we show that the conditions (I) and (II) of
  Lemma~\ref{lem:dimension-comparison} holds for
  $\k=\hdg^{ss}$. Condition (I) holds by
  Lemma~\ref{lem:one-factor}. To show that Condition (II) holds, by
  Lemma~\ref{lem:two-factors} it is enough to prove that for each
  $(a,b)\in G^2$ with $a\neq b$ and $a\neq \bar{b}$, there exists 
  $x\in G$ such that $h(xa)\neq h(xb)$. Suppose that this is not the
  case, then there exists a pair $(a,b)$ such that $h(xa)=h(xb)$ for
  all $x\in G$. Without loss of generality, we may and will assume
  that $b=1\in (\Z/q\Z)^*$, thus $a\neq \pm 1$.  It follows that
  $h(xa)=h(x)$ for all $x\in (\Z/q\Z)^*$.  Since $h$ is not a constant
  function, such an $a$ does not exists by Lemma~\ref{lem:even} of
  next section. Contradiction.
\end{proof}

\section{Arithmetic Results}
Throughout this section, $G=(\Z/q\Z)^*$. For each $a\in G$, let
$\theta_a:G\to G$ be the translation map: $b\mapsto ab$.  A function
$h:G\to \R$ is said to be \textit{even} if $h\circ \theta_{-1}=
h$. For any $x\leq y\in R$, we write $[x,y]_\Z$ for the set of
integers $\{i \mid x\leq i \leq y, \gcd(i, p)=1\}$.

\begin{lem}\label{lem:even}
  Let $h:(\Z/q\Z)^*\to \R$ be an even function that's monotonic on
  $[1,q/2]_\Z$. If $h\circ \theta_a=h$ for some $a\in
  (\Z/q\Z)^{*}$ and $a \neq \pm 1$, then $h$ is a constant
  function. 
\end{lem}
\begin{proof} We prove the Lemma in seven steps. 
\begin{step}\label{step:2}
  Let $\langle \pm a \rangle$ be the subgroup of $(\Z/q\Z)^{*}$
  generated by $a$ and $-1$. Clearly $h\circ \theta_b=h$ for any $b\in
  \langle \pm a \rangle$ since $h\circ \theta_a=h$ and $h$ is even. In
  particular, this holds true for the maximal element $b_{\max}$ in
  the set in $\langle \pm a \rangle \cap [1,q/2]_\Z$. If $b_{\max}=1$,
  the group $\langle \pm a \rangle$ is necessarily $\{ \pm 1
  \}$. Therefore, it is enough to prove that $h$ being nonconstant
  implies that $b_{\max}=1$. So with out lose of generality, we assume
  that $a=b_{\max}$ throughout the rest of the proof. Notice that if
  $a\neq 1$, then $2a^2>q$, for otherwise it contradicts the
  maximality of $a$. 
\end{step}
\begin{step}\label{lem:3}
Lemma~\ref{lem:even} holds if $p=2$. \\

Every even function on $(\Z/q\Z)^*$ is constant if $q$ is $2$ or $4$
so we assume that $q=2^r\geq 8$. The group $(\Z/2^r\Z)^* $ is
isomorphic to $ \Z/2\Z\times \Z/2^{r-2}\Z$, where the factor $\Z/2\Z$
is generated by $-1$. Let us assume that $\langle \pm a \rangle$ has
order $2^s$. Since $\langle \pm a \rangle\supseteq \langle \pm 1
\rangle$, it follows that $\langle \pm a \rangle \cong \Z/2\Z\times
\Z/2^{s-1}\Z$. In particular, if $\langle \pm a \rangle \neq \langle
\pm 1 \rangle$, then $\Z/2^{s-1}\Z$ is nontrivial, therefore $\langle
\pm a \rangle$ contains 3 elements of order two. But there are exactly
3 elements of order two in $(\Z/q\Z)^*: -1, 2^{r-1}-1,
2^{r-1}+1$. Hence $\langle\pm a\rangle$ contains all the above
elements of order $2$. So $a=2^{r-1}-1$ since it is the largest
element in $[1,q/2]_\Z$.  Therefore,
  \[h(q/2-1)=h(2^{r-1}-1)=h(a)=(h\circ \theta_a)(1)=h(1).\] Since $h$ is
  monotonic on $[1, q/2]_\Z$, the above equality implies that $h$
  is constant on $[1, q/2]_\Z$ and therefore a constant
  function.

\end{step}
\begin{step}\label{lem:4}
  Let $p$ be an odd prime. Lemma~\ref{lem:even} holds if either $a$
  is even, or $a$ is odd and $3a\geq q$. \\

  It is enough to prove that if $a\neq 1$, then
  $h(1)=h((q-1)/2)$. Since $h(1)=(h\circ
  \theta_a)(1)=h(a)$, by monotonicity $h$ is constant on
  $[1,a]_\Z$. Therefore it is enough to find $b$ such that
  $h((q-1)/2)=h(b)$ and $b\in [1,a]_\Z$. 

  First, let's assume that $a=2b$ is even. Then
  \[ a\cdot\frac{q-1}{2}= (q-1)b \equiv -b \mod q. \] So
  $h((q-1)/2)=h(a(q-1)/2)=h(-b)=h(b)$. Clearly $b=a/2$ lies in
  $[1,a]_\Z$.

Next, assume that $a$ is odd. Then
\[ a\cdot\frac{q-1}{2}=\frac{qa-a}{2} \equiv \frac{q-a}{2} \pmod q.\]
So $h((q-1)/2)=h((q-a)/2)$.  Let $b=(q-a)/2$.  When
$3a\geq q$, we have $ b=(q-a)/2 \leq a$ hence $b$ lies in $[1,a]_\Z$
as desired.
\end{step}

\begin{step}\label{lem:5}
Lemma~\ref{lem:even} holds if $p=3$.   \\

  When $p$ is odd, $(\Z/p^r\Z)^*$ is cyclic of order
  $\varphi(p^r)=(p-1)p^{r-1}$. For $p=3$,
  \[(\Z/3^r\Z)^* \cong \Z/(2\cdot 3^{r-1})\Z\cong \Z/2\Z\times
  \Z/3^{r-1}\Z.\] In particular, if $q\geq 9$, $(\Z/q\Z)^*$ contains a
  unique subgroup of order 3 which is generated by $3^{r-1}+1$. If the
  order of $\langle \pm a \rangle$ is coprime to $3$, then $\langle
  \pm a\rangle$ is necessarily $\{\pm 1\}$, which leads to an
  contradiction.  If the order of $\langle \pm a \rangle$ is divisible
  by $3$, then $q\geq 9$ and $\langle \pm a \rangle$ contains
  $3^{r-1}+1$. By assumption on the maximality of $a$ we must have $a
  \geq 3^{r-1}+1$ and hence $3a>q$.
\end{step}
\begin{step}\label{lem:6}
  Assume that both $p$ and $a$ are odd, $p \neq 3$ and
  $3a<q$. Lemma~\ref{lem:even} holds if $7a\geq q$. \\

  Since $p\neq 3$, $(q-3)/2$ lies in $[1, q/2]_\Z$. It is enough to prove
  that $a\neq 1$ implies that $h(1)=h((q-3)/2)$. Indeed, it follows from the
  proof of Step~\ref{lem:4} that $h((q-1)/2)=h((q-a)/2)$. But if $a\neq
  1$ then $a \geq 3$ so $(q-a)/2 \leq (q-3)/2$. If we prove that $h$
  is constant on $[1, (q-3)/2]_\Z$, then $h((q-1)/2)=h((q-a)/2)=h(1)$
  and it follows that $h$ is a constant function. 

  By our assumption $3a<q$, so $(q-3a)/2$ lies in $[1,
  q/2]_\Z$. Notice that
  \[ a\cdot\frac{q-3}{2} \equiv \frac{q-3a}{2} \mod q. \] We see that
  $h((q-3)/2)=h((q-3a)/2)$. Since $h$ is constant on $[1, a]_\Z$, the
  inequality $h(1) \neq h((q-3)/2)$ would imply that $ a <(q-3a)/2$,
  or equivalently $5a<q$.  In particular, $2a<q/2$. But $2\in [1,
  a]_\Z$ since $p$ is odd and $a\geq 3$. So $h(2)=h(1)$, therefore
  $h(2a)=h(1)$ and $h$ is constant on $[1, 2a]_\Z$. But now by our
  assumption $7a\geq q$, or equivalently $ 2a \geq (q-3a)/2$, it
  follows that
\[ h\left(\frac{q-3}{2}\right)=h\left(\frac{q-3a}{2}\right)=h(1).\]
 \end{step}
 \begin{step}\label{lem:7}
   Assume that both $p$ and $a$ are odd, $p \neq 3,5$ and
  $7a<q$. Lemma~\ref{lem:even} holds. \\

  Since $7a<q$ and $p\neq 5$, $(q-5a)/2$ lies in $[1, q/2]_\Z$. By
  similar argument as in Step \ref{lem:6}, $h((q-5)/2)=h((q-5a)/2)$.  We claim that now it is
  enough to show that $h(1)=h((q-5)/2)$. Indeed, by the proof of the
  Step~\ref{lem:6}, all we need to show is that $h(1)=h((q-3)/2)$, but
  since $a\geq 3$, then $(q-3a)/2< (q-5)/2$. So $h$ being constant on
  $[1, (q-5)/2]_\Z$ implies that $h(1)=h((q-3a)/2)=h((q-3)/2)$.

   Let $S$ be the set of all integers 
   \[ S=\{b \mid b \geq 1, p\nmid b, (2b+1)a<q\}. \] Clearly $1\in S$
   so $S$ is not empty. Let $x$ be the maximal element of $S$. By
   Step~\ref{step:2}, $2a^2>q$ so necessarily $x<a$. Since $h$ is
   constant on $[1, a]_\Z$, we must have $h(1)=h(x)$. Notice that $xa<
   q/2$ by assumption. So $h(ax)=h(x)=h(1)$ and it follows that $h$ is
   constant on $[1, ax]_\Z$. Assume that $h(1)\neq h((q-5)/2)$. It is
   necessary that $ax< (q-5a)/2$, or equivalently, $(2x+5)a<q$. But
    we can choose $x'$ from the two elements set $\{x+1,
   x+2\}$ such that $x'$ is coprime to $p$. It follows that $x'\in
   S$. This contradicts the maximality of $x$.
\end{step}
\begin{step}\label{lem:8}
Lemma~\ref{lem:even} holds if $p=5$.     \\

If the order of $\langle \pm a \rangle$ is divisible by
$5$, then $\langle \pm a \rangle$ contains the unique subgroup of
order 5 in $(\Z/5^r\Z)^*$. In particular, $2\cdot 5^{r-1}+1\in \langle
\pm a \rangle$.  It follows that $a> 2\cdot 5^{r-1}+1$ and therefore
$3a>5^r$. The Lemma holds by Step~\ref{lem:4}.

  If the order of $\langle \pm a \rangle$ is coprime to $5$. Then from
  the isomorphism
  \[ \Z/5^r\Z\cong \Z/4\Z\times \Z/5^{r-1}\Z,\] we see that $\langle
  \pm a \rangle$ is has either order $2$ or $4$. If $\langle \pm a
  \rangle$ has order 2, then $\langle \pm a \rangle$ is necessarily
 $\langle\pm 1\rangle$ and this leads to a contradiction. So we assume
  that $\langle \pm a \rangle$ has order $4$ and $a$ is the unique
  element such that $1<a<5^r/2$ and $a^2\equiv -1 \mod 5^r$.  In
  particular, $a^2+1\geq 5^r$. If $a$ is even then the Lemma holds by
  Step~\ref{lem:4}. In particular, this works for $q=p=5$ since $a=2$
  in this case. We assume that $q\geq 25$ and $a$ is odd through out
  the rest of the proof. First we claim that $a \geq 7$. Indeed, If
  $q=25$, then $a=7$ by direct calculation; if $q>25$, then $a>7$
  since $a^2+1\geq q$. This implies that $(q-a)/2\leq
  (q-7)/2$. Therefore, it is enough to prove that $h((q-7)/2)=h(1)$
  since it then follows that $h((q-1)/2)=h((q-a)/2)=h(1)$. By
  Step~\ref{lem:6} we may also assume that $7a<q$. It follows that
  $(q-7a)/2 \in [1, q/2]_\Z$ and $h((q-7)/2)=h((q-7a)/2)$.

  Let $c=[q/a]$. Since $a^2+1\geq q$ and $a<q/2$ we see that $2\leq
  c\leq a$. Let $x=[c/2]$ if $[c/2]$ is not divisible by 5, and 
  $x=[c/2]-1$ otherwise. Notice that $a>x\geq \max\{1, (c-3)/2\}$  and
  $xa \leq q/2$ by our choice of $x$. It follows that $x\in [1, a]_\Z$
  therefore $h(x)=h(1)$, and therefore $h(ax)=h(x)=h(1)$. So $h$ is
  constant on $[1, ax]_\Z$. If $h(1)\neq h((q-7)/2)$,  we must
  have $ xa < (q-7a)/2$, or equivalently, $(2x+7)a <q$.  Then it 
  follows that 
\[ \frac{q}{a} > 2x+7\geq 2\left(\frac{c-3}{2}\right)+7 =
c+4=\left[\frac{q}{a}\right]+4, \] 
which is absurd. 
  \end{step}

Lemma~\ref{lem:even} is proved by combining all the above steps. 
\end{proof}

\end{document}